\documentclass[a4paper,10pt]{amsart}

\usepackage[utf8]{inputenc} 
\usepackage{amsmath}
\usepackage{amsbsy}
\usepackage{amssymb}
\usepackage{amscd,amsthm}
\usepackage{verbatim}
\usepackage[english]{babel}
\usepackage{dsfont}
\usepackage{anysize}
\usepackage{hyperref}

\newcommand{\la}{\lambda}

\newcommand{\p}{\phantom}
\newcommand{\q}{\quad}

\newtheorem{thm}{Theorem}[section]

\theoremstyle{definition}

\theoremstyle{remark}

\title{Improving the Estimates for a Sequence \\ Involving Prime Numbers}
\author{Christian Axler}
\address{Institute of Mathematics\\ Heinrich-Heine University Düsseldorf\\
40225 Düsseldorf, Germany}
\email{christian.axler@hhu.de}
\date{\today}

\begin{document}

\begin{abstract}
Based on new explicit estimates for the prime counting function, we improve the currently known estimates for the particular sequence $C_n = np_n - \sum_{k \leq 
n}p_k$, $n \geq 1$, involving the prime numbers.
\end{abstract}

\maketitle

\section{Introduction}

Let $p_n$ denotes the $n$th prime number. In this paper, we establish new explicit estimates for the sequence $(C_n)_{n \geq 1}$ with
\begin{displaymath}
C_n = np_n - \sum_{k \leq n} p_k
\end{displaymath}
(see \cite{pol}). In \cite[Theorem 10]{axler2015}, the present author used the identity
\begin{equation}
C_n = \int_2^{p_n} \pi(x) \, dx, \tag{1} \label{1}
\end{equation}
where $\pi(x)$ denotes the number of primes not exceeding $x$, to derive that the asymptotic formula
\begin{equation}
C_n = \sum_{k=1}^{m-1} (k-1)! \left(1 - \frac{1}{2^k} \right) \frac{p_n^2}{\log^k p_n} + O \left( \frac{p_n^2}{\log^m p_n} \right). \tag{2} \label{2}
\end{equation}
holds for each positive integer $m$. By setting $m=9$ in \eqref{2}, we get
\begin{equation}
C_n = \frac{p_n^2}{2 \log p_n} + \frac{3p_n^2}{4\log^2 p_n} + \frac{7p_n^2}{4 \log^3 p_n} + \chi(n) + O \left( \frac{p_n^2}{\log^9 p_n} \right), \tag{3} 
\label{3}
\end{equation}
where
\begin{displaymath}
\chi(n) = \frac{45 p_n^2}{8\log^4 p_n} + \frac{93p_n^2}{4\log^5 p_n} + \frac{945p_n^2}{8 \log^6 p_n} + \frac{5715p_n^2}{8 \log^7 p_n} + \frac{80325p_n^2}{16 
\log^8 p_n}. 
\end{displaymath}
In the direction of \eqref{3}, the present author \cite[Theorem 3 and Theorem 4]{axler2015} showed that
\begin{equation}
C_n \geq \frac{p_n^2}{2 \log p_n} + \frac{3p_n^2}{4 \log^2 p_n} + \frac{7p_n^2}{4 \log^3 p_n} + \Theta(n) \tag{4} \label{4}
\end{equation}
for every $n \geq 52703656$, where
\begin{displaymath}
\Theta(n) = \frac{43.6p_n^2}{8\log^4 p_n} + \frac{90.9p_n^2}{4\log^5 p_n} + \frac{927.5p_n^2}{8\log^6 p_n} + \frac{5620.5p_n^2}{8\log^7 p_n} + 
\frac{79075.5p_n^2}{16\log^8 p_n}
\end{displaymath}
and that the upper bound
\begin{equation}
C_n \leq \frac{p_n^2}{2 \log p_n} + \frac{3p_n^2}{4 \log^2 p_n} + \frac{7p_n^2}{4 \log^3 p_n} + \Omega(n) \tag{5} \label{5}
\end{equation}
holds for every positive integer $n$, where
\begin{displaymath}
\Omega(n) = \frac{46.4p_n^2}{8\log^4 p_n} +  \frac{95.1p_n^2}{4\log^5 p_n} + \frac{962.5p_n^2}{8\log^6 p_n} + \frac{5809.5p_n^2}{8\log^7 p_n} + 
\frac{118848p_n^2}{16\log^8 p_n}.
\end{displaymath}
Using new explicit estimates for the prime counting function $\pi(x)$, which are found in \cite[Proposition 3.6 and Proposition 3.12]{axler2017}, we improve the 
inequalities \eqref{4} and \eqref{5} by showing the following both results.

\begin{thm} \label{thm1}
For every positive integer $n \geq 440200309$, we have
\begin{displaymath}
C_n \geq \frac{p_n^2}{2 \log p_n} + \frac{3p_n^2}{4 \log^2 p_n} + \frac{7p_n^2}{4 \log^3 p_n} + L(n),
\end{displaymath}
where
\begin{displaymath}
L(n) = \frac{44.4p_n^2}{8\log^4 p_n} + \frac{92.1p_n^2}{4\log^5 p_n} + \frac{937.5p_n^2}{8\log^6 p_n} + \frac{5674.5p_n^2}{8\log^7 p_n} + 
\frac{79789.5p_n^2}{16\log^8 p_n}.
\end{displaymath}
\end{thm}

\begin{thm} \label{thm2}
For every positive integer $n$, we have
\begin{displaymath}
C_n \leq \frac{p_n^2}{2 \log p_n} + \frac{3p_n^2}{4 \log^2 p_n} + \frac{7p_n^2}{4 \log^3 p_n} + U(n),
\end{displaymath}
where
\begin{displaymath}
U(n) = \frac{45.6p_n^2}{8\log^4 p_n} +  \frac{93.9p_n^2}{4\log^5 p_n} + \frac{952.5p_n^2}{8\log^6 p_n} + \frac{5755.5p_n^2}{8\log^7 p_n} + 
\frac{116371p_n^2}{16\log^8 p_n}.
\end{displaymath}
\end{thm}

\section{Preliminaries} \label{sec2}

In 1793, Gau{\ss} \cite{gauss} stated a conjecture concerning an asymptotic magnitude of $\pi(x)$, namely
\begin{equation}
\pi(x) \sim \text{li}(x) \q\q (x \to \infty), \tag{6} \label{6}
\end{equation}
where the \emph{logarithmic integral} $\text{li}(x)$ defined for every real $x \geq 0$ as
\begin{equation}
\text{li}(x) = \int_0^x \frac{dt}{\log t} = \lim_{\varepsilon \to 0} \left \{ \int_{0}^{1-\varepsilon}{\frac{dt}{\log t}} + 
\int_{1+\varepsilon}^{x}{\frac{dt}{\log t}} \right \} \approx \int_2^x \frac{dt}{\log t} + 1.04516 \ldots. \tag{7} \label{7}
\end{equation}
Using the method of integration of parts, \eqref{7} implies that
\begin{equation}
\text{li}(x) = \frac{x}{\log x} + \frac{x}{\log^2 x} + \frac{2x}{\log^3 x} + \frac{6x}{\log^4 x} + \frac{24x}{\log^5 x} + \ldots + \frac{(m-1)! x}{\log^mx}+ O 
\left( \frac{x}{\log^{m+1} x} \right) \tag{8} \label{8}
\end{equation}
for every positive integer $m$. The asymptotic formula \eqref{6} was proved independently by Hadamard \cite{hadamard1896} and by de la Vall\'{e}e-Poussin 
\cite{vallee1896} in 1896, and is known as the \textit{Prime Number Theorem}. By proving the existence of a zero-free region for the Riemann zeta-function 
$\zeta(s)$ to the left of the line $\text{Re}(s) = 1$, de la Vall\'{e}e-Poussin \cite{vallee1899} was able to estimate the error term in the Prime Number 
Theorem by
\begin{displaymath}
\pi(x) = \text{li}(x) + O(x \exp(-a\sqrt{\log x})),
\end{displaymath} 
where $a$ is a positive absolute constant. Together with \eqref{8}, we obtain that the asymptotic formula \begin{equation}
\pi(x) = \frac{x}{\log x} + \frac{x}{\log^2 x} + \frac{2x}{\log^3 x} + \frac{6x}{\log^4 x} + \frac{24x}{\log^5 x} + \ldots + \frac{(m-1)! x}{\log^mx}+ O \left( 
\frac{x}{\log^{m+1} x} \right). \tag{9} \label{9}
\end{equation}
holds for every positive integer $m$.

\section{A proof of Theorem \ref{thm1}} \label{sec3}

Now, we use some recent obtained lower bound for the prime counting function $\pi(x)$ to give a proof of Theorem \ref{thm1}.

\begin{proof}[Proof of Theorem \ref{thm1}]
First, let $m$ be a positive integer with $m \geq 2$, and let $a_2, \ldots, a_m$, $x_0$, and $y_0$ be real numbers, so that
\begin{equation}
\pi(x) \geq \frac{x}{\log x} + \sum_{k=2}^m \frac{a_kx}{\log^k x} \tag{10} \label{10}
\end{equation}
for every $x \geq x_0$ and
\begin{equation}
\text{li}(x) \geq \sum_{j=1}^{m-1} \frac{(j-1)! x}{\log^j x} \tag{11} \label{11}
\end{equation}
for every $x \geq y_0$. The asymptotic formulae \eqref{9} and \eqref{8} guarantee the existence of such parameters. In \cite[Theorem 13]{axler2015}, the present 
author showed that
\begin{equation}
C_n \geq d_0 + \sum_{k=1}^{m-1} \left( \frac{(k-1)!}{2^k} ( 1 + 2t_{k-1,1} ) \right) \frac{p_n^2}{\log^k p_n} \tag{12} \label{12}
\end{equation}
for every $n \geq \max \{ \pi(x_0) + 1, \pi(\sqrt{y_0}) + 1 \}$, where $t_{i,j}$ is defined by
\begin{equation}
t_{i,j} = (j-1)! \sum_{l=j}^{i} \frac{2^{l-j}a_{l+1}}{l!}. \tag{13} \label{13}
\end{equation}
and $d_0$ is given by
\begin{displaymath}
d_0 = d_0(m,a_2,\ldots, a_m, x_0) = \int_2^{x_0} \pi(x) \, dx - ( 1 + 2 t_{m-1,1} )\, \text{li}(x_0^2) + \sum_{k=1}^{m-1} t_{m-1,k} \frac{x_0^2}{\log^k
x_0}.
\end{displaymath}
Now, we choose $m=9$, $a_2=1$, $a_3 = 2$, $a_4=5.85$, $a_5=23.85$, $a_6 = 119.25$, $a_7 = 715.5$, $a_8 = 5008.5$, $a_9 = 0$, $x_0 = 19027490297$ and $y_0 = 
4171$. By \cite[Proposition 3.12]{axler2017}, we obtain that the inequality \eqref{10} holds for every $x \geq x_0$ and \eqref{11} holds for every $x \geq y_0$ 
by \cite[Lemma 15]{axler2015}. Substituting these values in \eqref{12}, we get
\begin{displaymath}
C_n \geq d_0 + \frac{p_n^2}{2 \log p_n} + \frac{3p_n^2}{4 \log^2 p_n} + \frac{7p_n^2}{4 \log^3 p_n} + L(n)
\end{displaymath}
for every $n \geq 841160647 = \pi(x_0)$, where $d_0 = d_0(9,1,2,5.85,23.85,119.25,715.5,5008.5,0, x_0)$ is given by
\begin{align}
d_0 & = \int_2^{x_0} \pi(x) \, dx - 253.3 \, \text{li}(x_0^2) + \frac{126.15 x_0^2}{\log x_0} + \frac{62.575 x_0^2}{\log^2 x_0} + \frac{61.575 x_0^2}{\log^3 
x_0} \nonumber \\
& \p{\q\q} + \frac{89.4375x_0^2}{\log^4 x_0} + \frac{165.95x_0^2}{\log^5 x_0} + \frac{357.75x_0^2}{\log^6 x_0} + \frac{715.5x_0^2}{\log^7 x_0}. \tag{14} 
\label{14}
\end{align}
The present author \cite[Lemma 16]{axler2015} found that
\begin{displaymath}
\text{li}(x) \leq \frac{x}{\log x} + \frac{x}{\log^2 x} + \frac{2x}{\log^3 x} + \frac{6x}{\log^4 x} + \frac{24x}{\log^5 x} + \frac{120x}{\log^6 x} + 
\frac{900x}{\log^7 x}
\end{displaymath}
for every $x \geq 10^{16}$. Applying this inequality to \eqref{14}, we get
\begin{align*}
d_0 & \geq \int_2^{x_0} \pi(x) \, dx - \frac{x_0^2}{2 \log x_0} - \frac{3 x_0^2}{4 \log^2 x_0} - \frac{7 x_0^2}{4 \log^3 x_0} - \frac{5.55x_0^2}{\log^4 x_0} - 
\frac{23.025 x_0^2}{ \log^5 x_0} \\
& \p{\q\q} - \frac{117.1875x_0^2}{\log^6 x_0} - \frac{1065.515625x_0^2}{\log^7 x_0}.
\end{align*}
Computing the right-hand side of the last inequality, we get
\begin{equation}
d_0 \geq \int_2^{x_0} \pi(x) \, dx - 8.188366 \cdot 10^{18}. \tag{15} \label{15}
\end{equation}
Since $x_0 = p_{841160647}$, we use \eqref{1} and a computer to obtain
\begin{displaymath}
\int_2^{x_0} \pi(x) \, dx = C_{841160647} = 8188378036394419009.
\end{displaymath}
Hence, by \eqref{15}, we get $d_0 \geq 1.12 \cdot 10^{13} > 0$. So we obtain the desired inequality for every $n \geq 841160647$. For every $440200309 \leq n 
\leq 841160646$ we check the inequality with a computer.
\end{proof}

\section{A proof of Theorem \ref{thm2}} \label{sec4}

Next, we use a recent result concerning an upper bound for the prime counting function $\pi(x)$ to establish the required inequality stated in Theorem 
\ref{thm2}.

\begin{proof}[Proof of Theorem \ref{thm2}]
Let $m$ be a positive integer with $m \geq 2$, let $a_2, \ldots, a_m, x_1$ be real numbers so that
\begin{equation}
\pi(x) \leq \frac{x}{\log x} + \sum_{k=2}^m \frac{a_kx }{\log^k x} \tag{16} \label{16}
\end{equation}
for every $x \geq x_1$ and let $\la, y_1$ be real numbers so that
\begin{equation}
\text{li}(x) \leq \sum_{j=1}^{m-2} \frac{(j-1)! x}{\log^j x} + \frac{\la x}{\log^{m-1} x} \tag{17} \label{17}
\end{equation}
for every $x \geq y_1$. Again, the asymptotic formulae \eqref{9} and \eqref{8} guarantee the existence of such parameters. The present author \cite[Theorem 
14]{axler2015} found that the inequality
\begin{align}
C_n & \leq d_1 + \sum_{k=1}^{m-2} \left( \frac{(k-1)!}{2^k} ( 1 + 2t_{k-1,1} ) \right) \frac{p_n^2}{\log^k p_n} \nonumber \\
& \phantom{\quad\quad} + \left( \frac{(1 + 2t_{m-1,1})\la}{2^{m-1}} - \frac{a_m}{m-1} \right) \frac{p_n^2}{\log^{m-1} p_n} \tag{18} \label{18}
\end{align}
holds for every $n \geq \max \{ \pi(x_1) + 1, \pi(\sqrt{y_1}) + 1 \}$, where $t_{i,j}$ is defined by \eqref{13}, and 
\begin{displaymath}
d_1 = d_1(m,a_2, \ldots, a_m, x_1) =  \int_2^{x_1} \pi(x) \, dx - ( 1 + 2 t_{m-1,1} ) \, \text{li}(x_1^2) + \sum_{k=1}^{m-1} t_{m-1,k} \frac{x_1^2}{\log^k x_1}.
\end{displaymath}
Next, we choose $m = 9$, $a_2=1$, $a_3=2$, $a_4=6.15$, $a_5=24.15$, $a_6=120.75$, $a_7=724.5$,$a_8=6601$, $a_9 = 0$, $\lambda=6300$, $x_1 = 13$ and $y_1 = 
10^{18}$. By \cite[Proposition 3.6]{axler2017}, we get that the inequality \eqref{16} holds for every $x \geq x_1$ and by \cite[Lemma 19]{axler2015}, that 
\eqref{17} holds for every $y \geq y_1$. By substituting these values \eqref{18}, we get
\begin{equation}
C_n \leq d_1 + \frac{p_n^2}{2 \log p_n} + \frac{3p_n^2}{4 \log^2 p_n} + \frac{7p_n^2}{4 \log^3 p_n} +  U(n) - \frac{0.375p_n^2}{16 \log^8 p_n} \tag{19} 
\label{19}
\end{equation}
for every $n \geq 50847535$, where $d_1 = d_1(9, 1, 2, 6.15, 24.15, 120.75, 724.5, 6601, 0, x_1)$ is given by
\begin{align*}
d_1 & = \int_2^{x_1} \pi(x) \, dx - \frac{26599}{90} \; \text{li}(x_1^2) + \frac{26509 x_1^2}{180 \log x_1} + \frac{26329 x_1^2}{360 \log^2 x_1} + \frac{25969 
x_1^2}{360 \log^3 x_1} \\
& \p{\q\q} + \frac{25231 x_1^2}{240 \log^4 x_1} + \frac{11891 x_1^2}{60 \log^5 x_1} + \frac{5221 x_1^2}{12 \log^6 x_1} + \frac{943 x_1^2}{\log^7 x_1}.
\end{align*}
A computation shows that $d_1 \leq 453$. We define
\begin{displaymath}
f(x) = \frac{0.375x^2}{16\log^8 x} - 453.
\end{displaymath}
Since $f(9187322) > 0$ and $f'(x) \geq 0$ for every $x \geq e^4$, we get $f(p_n) \geq 0$ for every $n \geq \pi(9187322) + 1 = 614124$. Now we can use \eqref{19} 
to obtain the desired inequality for every positive integer $n \geq 50847535$. Finally, we check the remaining cases with a computer.
\end{proof}

\end{document}